\documentclass[11pt,reqno]{article}
\usepackage[utf8]{inputenc}
\usepackage{enumerate}
\usepackage{lmodern}
\usepackage[T1]{fontenc}
\usepackage{verbatim}
\usepackage{textcomp}

\usepackage[english]{babel}
\usepackage[a4paper,vmargin={3.5cm,3.5cm},hmargin={2.5cm,2.5cm}]{geometry}
\usepackage[font=sf, labelfont={sf,bf}, margin=1cm]{caption}
\usepackage[pdftex]{hyperref}
\usepackage[pdftex]{color,graphicx}
\usepackage{color}

\usepackage{amsmath,amsfonts,amssymb,amsthm,mathrsfs}

\newcommand{\blue}{\color{black}}

\newtheorem{theorem}{Theorem}[section]
\newtheorem{corollary}[theorem]{Corollary}
\newtheorem{lemma}[theorem]{Lemma}
\newtheorem{p}[theorem]{Proposition}

\theoremstyle{remark}
\newtheorem{remark}[theorem]{Remark}
\theoremstyle{definition}

\def\RR{\mathbb{R}}\def\R{\mathbb{R}}
\def\HH{\mathfrak{H}}
\def\ZZ{\mathbb{Z}}
\def\NN{\mathbb{N}}

\def\PP{\mathbb{P}}
\def\EE{\mathbb{E}}
\def\Var{\mathbb{V}\mathrm{ar}}
\def\Cov{\mathbb{C}\mathrm{ov}}

\def\ga{\gamma}
\def\de{\delta}

\def\ph{\varphi}

\def\wt{\widetilde}
\def\wh{\widehat}

\def\dh2l{\mathbf{d}_{\mathbb{H}_{2\ell}}}
\def\d2{d_2}

\def\dc{d_\mathrm{c}}

\newcommand{\bra}[1]{\left\lbrace#1\right\rbrace}

\newcommand{\norm}[1]{\left\lVert#1\right\rVert}
\newcommand{\op}[1]{\left\lVert#1\right\rVert_\mathrm{op}}

\begin{document}
\title{ Multivariate normal approximation on the Wiener space: \\ new bounds in the convex distance}
\author{Ivan Nourdin\thanks{University of Luxembourg, 
Department of Mathematics, inourdin@gmail.com}, Giovanni Peccati\thanks{University of Luxembourg, 
Department of Mathematics, giovanni.peccati@gmail.com} and Xiaochuan Yang\thanks{University of Bath, Department of Mathematical Sciences, xiaochuan.j.yang@gmail.com}}
\date{\small\today}
\maketitle

\begin{abstract} 
We establish explicit bounds on the convex distance between the distribution of a vector of smooth functionals of a Gaussian field, and that of a normal vector with a positive definite covariance matrix. Our bounds are commensurate to the ones obtained by Nourdin, Peccati and R\'eveillac (2010) for the (smoother) 1-Wasserstein distance, and do not involve any additional logarithmic factor. One of the main tools exploited in our work is a recursive estimate on the convex distance recently obtained by Schulte and Yukich (2019). We illustrate our abstract results in two different situations: (i) we prove a quantitative multivariate fourth moment theorem for vectors of multiple Wiener-It\^o integrals, and (ii) we characterise the rate of convergence for the finite-dimensional distributions in the functional Breuer-Major theorem.

\smallskip


\smallskip

\end{abstract}

\section{Introduction}

Fix $m\geq 1$, and consider random vectors ${\bf F}$ and ${\bf G}$ with values in $\R^m$. The {\bf convex distance} between the distributions of ${\bf F}$ and ${\bf G}$ is defined as 
\begin{equation}\label{e:cd}
\dc(\mathbf F, \mathbf G) := \sup_{h\in\mathcal I_m} \big|\EE h(\mathbf F)  -\EE h(\mathbf G)\big|,
\end{equation}
where the supremum runs over the class $\mathcal I_m$ of indicator functions of the measurable convex subsets of $\RR^m$. For $m\geq 2$, the distance $d_c$ represents a natural counterpart to the well-known {\bf Kolmogorov distance} on the class of probability distributions on the real line, and enjoys a number of desirable invariance properties that make it well-adapted to applications\footnote{\blue For instance, one has that $ \dc( T  \mathbf F, T\mathbf G) = \dc(\mathbf F, \mathbf G)$, whenever the mapping $T : \RR^m \longrightarrow \RR^m$ is an invertible affine mapping --- see e.g. \cite{bentkus, R19} for more details.}.
 
 \smallskip
 
 The aim of the present note is to establish explicit bounds on the quantity $\dc(\mathbf F, \mathbf G)$, in the special case where ${\bf F}$ is a vector of smooth functionals of an infinite-dimensional Gaussian field, and ${\bf G}  = N_\Sigma$ is a $m$-dimensional centered Gaussian vector with covariance $\Sigma>0$. Our main tool is the so-called {\bf Malliavin-Stein method} for probabilistic approximations \cite{NP}, that we will combine with some powerful recursive estimates on $d_c$, recently derived in \cite{SY19} in the context of multidimensional second-order Poincar\'e inequalities on the Poisson space --- see Lemma \ref{l:SY2.4} below. 
 
 \smallskip
 
 Multidimensional normal approximations in the convex distance have been the object of an intense study since several decades, mostly in connection with multivariate central limit theorems (CLTs) for sums of independent random vectors --- see e.g. \cite{bentkus, G, nagaev, nazarov} for some classical references, as well as \cite{R19} for recent advances and for a discussion of further relevant literature. The specific challenge we are setting ourselves in the present work is to establish bounds on the quantity $\dc({\bf F}, N_{\Sigma})$ that coincide (up to an absolute multiplicative constant) with the bounds deduced in \cite{NPR10} on the {\bf 1-Wasserstein distance}
\begin{equation} \label{e:defwass}
d_W({\bf F}, N_{\Sigma}) := \sup_{h\in {\rm Lip}(1)} \left| \EE h(\mathbf F) - \EE h(N_\Sigma)\right|,
\end{equation}
where ${\rm Lip}(1)$ denotes the class of $C^1$ mappings $h:\RR^m\to\R$ with Lipschitz constant not exceeding 1. We will see that our estimates systematically improve the bounds that one can infer from the general inequality
\begin{equation} \label{e:conwass}
d_c({\bf F}, N_{\Sigma}) \leq K \, \sqrt{d_W({\bf F}, N_{\Sigma})},
\end{equation}
where $K$ is an absolute constant uniquely depending on $\Sigma$. For the sake of completeness, a full proof of \eqref{e:conwass} is presented in Appendix \ref{s:conwass}, where one can also find more details on the constant $K$.

\begin{remark}{\rm In order for the quantity $d_W({\bf F}, N_{\Sigma})$ to be well-defined, one needs that $\EE\|{\bf F}\|_{\RR^m} < \infty$. In Appendix \ref{s:conwass} we will also implicitly use the well-known representation
$$
d_W({\bf F}, N_{\Sigma})  = \inf_{ ({\bf U},{\bf V}) } \EE\left\|  {\bf U} - {\bf V} \right\|_{\RR^m},
$$
where the infimum runs over all couplings $({ \bf U},{\bf V})$ of ${\bf F}$ and $N_\Sigma$. See \cite[Ch. I-6]{V} for a proof of this fact, and for further relevant properties of Wasserstein distances.
}
\end{remark}

The main contributions of our paper are described in full detail in Section \ref{ss:intromain} and Section \ref{ss:introapp}. Section \ref{ss:intromall} contains some elements of Malliavin calculus that are necessary in order to state our findings. Section \ref{ss:introd2} discusses some estimates on the smooth distance $d_2$ (to be defined therein) that can be obtained by {interpolation techniques}, whereas Section \ref{ss:introwass} provides an overview of the main results of \cite{NPR10}.

\smallskip

\noindent{\it Remark on notation.} From now on, every random element is assumed to be defined on a common probability space $(\Omega, \mathscr{F}, \PP)$, with $\EE$ denoting expectation with respect to $\PP$. For $p\geq 1$, we write $L^p(\Omega) := L^p (\Omega, \mathscr{F}, \PP)$.

\subsection{Elements of Malliavin calculus}\label{ss:intromall}

The reader is referred e.g. to the monographs \cite{NP, nualartbook, eulalia} for a detailed discussion of the concepts and results presented in this subsection. 

\smallskip

 Let $\frak H$ be a real separable Hilbert space, and write $\langle \cdot, \cdot \rangle_\frak{H}$ for the corresponding inner product. In what follows, we will denote by $X=\{X(h) : h\in\frak H\}$ an {\bf isonormal Gaussian process} over $\frak H$, that is, $X$ is a centered Gaussian family indexed by the elements of $\frak H$ and such that $\EE [X(h)X(g)] = \langle h,g\rangle_\frak{H}$ for every $h,g \in \frak H$. For the rest of the paper, we will assume without loss of generality that $\mathscr{F}$ coincides with the $\sigma$-field generated by $X$.
 
\smallskip

Every $F\in L^2(\Omega)$ admits a {\bf Wiener-It\^o chaos expansion} of the form
\begin{align*}
F = \EE F +  \sum_{q=1}^\infty I_q(f_q),
\end{align*}
where $f_q$ belongs to the symmetric $q$th tensor product $\frak H^{\odot q}$ (and is uniquely determined by $F$), and  $I_q(f_q)$ is the $q$-th {\bf multiple Wiener-It\^o integral} of $f_q$ with respect to $X$. One writes $F\in \mathbb D^{1,2}(\Omega)$ if 
\begin{align*}
\sum_{q\ge 1} q q!\norm{f_q}_{\frak{H}^{\otimes q}}^2<\infty.
\end{align*}
For $F\in \mathbb D^{1,2}(\Omega)$, we denote by $DF$ the {\bf Malliavin derivative} of $F$, and recall that $DF$ is by definition a random element with values in $\frak H$. The operator $D$ satisfies a fundamental {\bf chain rule}: if $\varphi:\RR^m\to\RR$ is $C^1$ and has bounded derivatives and if $F_1,\dots,F_m\in\mathbb D^{1,2}(\Omega)$, then $\varphi(F_1,\ldots,F_m)\in\mathbb D^{1,2}(\Omega)$ and 
\begin{equation}\label{chainrule}
D\varphi(F_1,\ldots,F_m) = \sum_{i=1}^m \partial_i\varphi(F_1,\ldots,F_m)DF_i.
\end{equation}
For general $p>2$, one writes $F\in\mathbb D^{1,p}(\Omega)$ if $F\in L^p(\Omega)\cap \mathbb D^{1,2}(\Omega) $ and $\EE[\norm{DF}_\frak{H}^p]<\infty$.  

\smallskip

The adjoint of $D$, customarily called the {\bf divergence operator} or the {\bf Skorohod integral}, is denoted by $\de$ and satisfies the duality formula,
\begin{align}\label{duality}
\EE[\de(u)F] = \EE[\langle u, DF\rangle_\frak{H}]
\end{align} 
for all $F\in\mathbb D^{1,2}(\Omega)$, whenever $u: \Omega\to \frak H$ is in the domain ${\rm Dom}(\delta)$ of $\de$. 

\smallskip

The {\bf generator of the Ornstein-Uhlenbeck semigroup}, written $L$, is defined by the relation
$
L F= - \sum_{q\ge 1} q I_q(f_q)
$
for every $F$ such that $\sum_{q\ge 1} q^2 q! \norm{f_q}_{\frak H^{\otimes q}}^2<\infty$. The {\bf pseudo-inverse} of $L$, denoted by $L^{-1}$, is the operator defined for any $F\in L^2(\Omega)$ as
$
L^{-1} F = - \sum_{q\ge 1} \frac{1}{q} I_q(f_q). 
$
The crucial relation that links the objects introduced above is the identity
\begin{align}
\label{e:relation}
F= \EE F -\delta(DL^{-1}F),
\end{align}
which is valid for any $F\in L^2(\Omega)$ (in particular, one has that, for every $F\in L^2(\Omega)$, $DL^{-1}F \in {\rm Dom} (\delta)$).

\subsection{Bounds on the smooth distance $d_2$}\label{ss:introd2}

Fix $m\geq 1$ and assume that $\mathbf F=(F_1,...,F_m)$ is a centered random vector in $\RR^m$ whose components belong to $\mathbb D^{1,2}(\Omega)$. 
Without loss of generality, we can assume that each $F_i$ is of the form $F_i=\delta(u_i)$ for some $u_i\in{\rm Dom}(\delta)$; indeed, by virtue of (\ref{e:relation}) one can always set $u_i=-DL^{-1}F_i$ (although this is by no means the only possible choice). Let also $N_\Sigma=(N_1,\ldots,N_m)$ be a centered Gaussian vector with invertible $m\times m$ covariance matrix $\Sigma= \{\Sigma(i,j) \} =  \{\Sigma(i,j) : i,j = 1,...,m \}$.
Finally, consider the so-called {\bf $d_2$ distance} (between the distributions of ${\bf F}$ and $N_\Sigma$) defined by
$$d_2(\mathbf F, N_\Sigma) = \sup_{h}\big|\EE h(\mathbf F)  -\EE h(N_\Sigma)\big|,$$
where the supremum is taken over all $C^2$ functions $h:\RR^m\to\R$  that are {\blue 1-Lipschitz and} such that 
$\sup_{{\bf x}\in\RR^m}\|({\rm Hess}\,h)({\bf x})\|_{\rm H.S.}\leq 1$;
here, ${\rm Hess}\, h$ stands for the Hessian matrix of $h$, whereas $\|\cdot \|_{\rm H.S.}$ \textcolor{black}{(resp. $\langle\cdot,\cdot\rangle_{\rm H.S.}$)} denotes the Hilbert-Schmidt norm \textcolor{black}{(resp. scalar product)}, that is, \textcolor{black}{$\langle A,B\rangle_{\rm H.S.}={\rm Tr(AB^T)}=\sum_{1\leq i,j\leq m}A(i,j)B(i,j)$ and $\|A\|^2_{\rm H.S.}=\langle A,A\rangle_{\rm H.S.}$ for any $m\times m$ matrices $A=\{A(i,j)\}$ and $B=\{B(i,j)\}$}.

\smallskip

For a given $h:\R^m\to\R\in C^2$ with bounded partial derivatives, let us introduce its mollification at level $\sqrt{t}$,  defined by
\begin{align}\label{e:h_t}
h_t(\mathbf x) :=  \EE[h(\sqrt{t} N_\Sigma + \sqrt{1-t} \mathbf x)],\quad {\bf x}\in\R^m.
\end{align} 
One has
\begin{eqnarray*}
&&\EE h(N_\Sigma) - \EE h(\mathbf F) = \int_0^1 \frac{d}{dt} \EE[h_t( {\bf F})]dt\\
&=&\frac12\sum_{i=1}^m \int_0^1 \left\{\frac{1}{\sqrt{t}}\EE [\partial_i h(\sqrt{t}N_\Sigma + \sqrt{1-t} \mathbf F)N_i]
- \frac{1}{\sqrt{1-t}}\EE [\partial_i h(\sqrt{t}N_\Sigma + \sqrt{1-t} \mathbf F)F_i]\right\}dt.
\end{eqnarray*} 
\textcolor{black}{Supposing in addition (and without loss of generality) that ${\mathbf F}$ and $N_\Sigma$ are independent, we can write, using the Gaussian integration by parts
\begin{eqnarray*}
\EE [\partial_i h(\sqrt{t}N_\Sigma + \sqrt{1-t} \mathbf F)N_i]&=&E^{\mathbf F} \big(E^{N_\Sigma}[\partial_i h(\sqrt{t}N_\Sigma + \sqrt{1-t} \mathbf F)N_i]\big)\\
&=&\sqrt{t}\sum_{j=1}^m  \EE [\partial^2_{ij} h(\sqrt{t}N_\Sigma + \sqrt{1-t} \mathbf F)]\EE[N_iN_j],
\end{eqnarray*} 
and, combining the duality formula (\ref{duality}) with the chain rule (\ref{chainrule}),
\begin{eqnarray*}
\EE [\partial_i h(\sqrt{t}N_\Sigma + \sqrt{1-t} \mathbf F)F_i]&=&E^{N_\Sigma} \big(E^{\mathbf F}[\partial_i h(\sqrt{t}N_\Sigma + \sqrt{1-t} \mathbf F)F_i]\big)\\
&=&\sqrt{1-t}\sum_{j=1}^m  \EE [\partial^2_{ij} h(\sqrt{t}N_\Sigma + \sqrt{1-t} \mathbf F)\langle DF_i,u_j\rangle_{\HH}].
\end{eqnarray*} 
Putting everything together leads to }
$$\EE h(N_\Sigma) - \EE h(\mathbf F) =\frac12 \int_0^1 \EE[\langle({\rm Hess}\,h)(\sqrt{t}N_\Sigma + \sqrt{1-t} \mathbf F), \Sigma - M_F\rangle_{\rm H.S.}]dt,
$$
where $M_F$ is the random $m\times m$ matrix given by 
\begin{equation}\label{mf}
M_F(i,j)=\langle DF_i,u_j\rangle_{\HH}.
\end{equation}
It is 
then immediate  that 
\begin{equation}\label{e:bound_d2}
d_2(\mathbf F, N_\Sigma) \leq \frac12\,\EE\|M_F-\Sigma\|_{\rm H.S.}.
\end{equation}
Inequalities in the spirit of \eqref{e:bound_d2} were derived e.g. in \cite{NPRei10} (in the context of limit theorems for homogeneous sums) and \cite{PZ} (in the framework of multivariate normal approximations on the Poisson space) --- see also \cite{SY19} and the references therein.

\subsection{Bounds on the 1-Wasserstein distance}\label{ss:introwass}

For random vectors ${\bf F}$ and $N_\Sigma$ as in the previous section, we will now discuss a suitable method for assessing the quantity $d_W(\mathbf F,N_\Sigma)$ defined in \eqref{e:defwass}, that is, for uniformly bounding the absolute difference $ | \EE h(\mathbf F) - \EE h(N_\Sigma) | $ over all 1-Lipschitz functions $h$ of class $C^1$.  

\smallskip

Since we do not assume  $h$ to be twice differentiable, the method presented in Section \ref{ss:introd2} no longer works.
A preferable approach is consequently the so-called `Malliavin-Stein method', introduced in \cite{NP09} in dimension 1, and later extended to the multivariate setting in \cite{NPR10}. 
Let us briefly recall how this works (see \cite[Chapter 4 and Chapter 6]{NP} for a full discussion, and \cite{W} for a constantly updated list of references). 

\smallskip

Start by considering the following {\bf Stein's equation}, with $h:\RR^m\to\RR$ given and  $f:\RR^m\to\RR$ unknown:
\begin{align}
\label{e:SE}
\sum_{i,j=1}^m \Sigma(i,j)\partial^2_{ij} f(\mathbf x) -\sum_{i=1}^m x_i \partial_i f(\mathbf x) = h(\mathbf x) - \EE h(N_\Sigma), \quad {\bf x} \in \RR^m.
\end{align}
When $h\in C^1$ has bounded partial derivatives, it turns out that \eqref{e:SE} admits a solution $f=f_h$ of class $C^2$ and whose second partial derivatives are bounded --- see e.g. \cite[Proposition 4.3.2]{NP} for a precise statement. Taking expectation with respect to the distribution of $\mathbf F$ in \eqref{e:SE} gives
\begin{align*}
\EE h(\mathbf F) - \EE h(N_\Sigma) &= \sum_{i,j=1}^m \Sigma(i,j) \EE[\partial^2_{ij} f_h(\mathbf F)] -\sum_{i=1}^m \EE[F_i \partial_i f_h(\mathbf F)].
\end{align*}
We can apply again the duality formula (\ref{duality}) together with the chain rule (\ref{chainrule}), to deduce that
\begin{equation}\label{e:h_t2}
\EE h(\mathbf F) - \EE h(N_\Sigma) 
= \EE[\langle({\rm Hess}\,f_h)(\mathbf F),M_F-\Sigma \rangle_{\rm H.S.}],
\end{equation}
where $M_F$ is defined in (\ref{mf}). Taking the supremum over the set of all 1-Lipschitz functions $h:\RR^m\to\RR$ of class $C^1$, we infer
\begin{align}\label{e:bound_W1}
d_W(\mathbf F,N_\Sigma)  \leq c_1\, \EE\|M_F-\Sigma\|_{\rm H.S.},
\end{align}
with 
\begin{equation}\label{e:wc}
c_1= \sup_{h\in {\rm Lip}(1)}\sup_{{\bf x}\in\RR^m}\|({\rm Hess}\,f_h)({\bf x})\|_{\rm H.S.}\leq \sqrt{m}  \, \|\Sigma^{-1}\|_{\rm op} \, \|\Sigma\|_{\rm op}^{1/2},
\end{equation}
and $\op{\cdot}$ is the operator norm for $m\times m$ matrices. The estimate \eqref{e:bound_W1} is the main result of \cite{NPR10} (see also \cite[Theorem 6.1.1]{NP}), whereas a self-contained proof of \eqref{e:wc} can be found in \cite[Proposition 4.3.2]{NP}. 

\subsection{Main results: bounds on the convex distance}\label{ss:intromain}

The principal aim of the present paper is to address the following natural question: {\it can one obtain a bound similar to {\rm (\ref{e:bound_W1})} for distances based on {\bf non-smooth} test functions $h:\RR^m\to\RR$, such as e.g.
{\it indicator functions} of measurable convex subsets of $\RR^m$? }

\smallskip

If $h$ is such an indicator function, then we recall e.g. from \cite[Lemma 2.2]{SY19} that, for all $t\in(0,1)$,
\begin{align*}
\big|\EE h(\mathbf F) - \EE h(N_\Sigma) \big| \le \frac{4}{3} |\EE h_t(\mathbf F) -\EE h_t(N_\Sigma) | + \frac{20m}{\sqrt{2}}\frac{\sqrt{t}}{1-t},
\end{align*}
where $h_t$ stands for the mollification at level $\sqrt{t}$ of $h$, as defined in (\ref{e:h_t}).
Let $f_t=f_{h_t}$ be the solution of the Stein's equation (\ref{e:SE}) associated with $h=h_t$. In \cite{CGS} (see also \cite{SY19}), it is shown that
\begin{align*}
\max_{1\leq i,j\leq m}\sup_{{\bf x}\in\RR^m} |\partial^2_{ij} f_t(\mathbf x) | &\le c_2  |\log t|,
 \end{align*} 
 with $c_2=c_2(m,\Sigma)$ a constant depending only on $m$ and $\Sigma$.
 Combining such an estimate \textcolor{black}{with (\ref{e:h_t2})} yields the existence of a constant $c_3=c_3(m,\Sigma)>0$ such that
 \begin{equation}\label{smoothing}
 \big|\EE h(\mathbf F) - \EE h(N_\Sigma) \big| \le c_3\left( \EE\|M_F-\Sigma\|_{\rm H.S.} |\log t|+ \frac{\sqrt{t}}{1-t}\right).
 \end{equation}
 From (\ref{smoothing}),
 it is straightforward to deduce the existence of $c_4=c_4(m,\Sigma)>0$  such that
 \begin{equation}\label{easy}
\dc(\mathbf F,N_\Sigma)  \le c_4\, \EE\|M_F-\Sigma\|_{\rm H.S.}\,\big|\log \{  \EE\|M_F-\Sigma\|_{\rm H.S.}\}\big|.
 \end{equation}
Comparing (\ref{easy}) with (\ref{e:bound_d2}) and \eqref{e:bound_W1} shows that such a strategy yields a bound on $\dc(\mathbf F,N_\Sigma)$ differing from those deduced above for the distances $d_2$ and $d_W$ by an additional logarithmic factor. See also \cite{KP15,NPS14} for more inequalities analogous to (\ref{easy})  --- that is, displaying a multiplicative logarithmic factor --- related respectively to the (multivariate) Kolmogorov and total variation distances.

 \medskip

In this paper, we will show that one can actually remove the redundant logarithmic factor
on the right-hand side of (\ref{easy}), thus yielding a bound on $\dc({\bf F}, N_\Sigma)$ that is commensurate to (\ref{e:bound_d2}) and \eqref{e:bound_W1}  (with moreover an explicit multiplicative constant). Our main result is the following:

\begin{theorem}\label{t} Let $\mathbf F=(F_1,...,F_m)=(\delta(u_1),\ldots,\delta(u_m))$ be a vector  in $\R^m$ of centered random variables such that $u_i\in{\rm Dom}(\delta)$, for $i=1,...,m$. 
Let also $N_\Sigma=(N_1,\ldots,N_m)$ be a centered Gaussian vector with invertible $m\times m$ covariance matrix $\Sigma=\{\Sigma(i,j)\}$.
Then
$$\dc(\mathbf F,N_\Sigma)\le 402 \Big(\op{\Sigma^{-1}}^{3/2}+1\Big) m^{41/24} \, \textcolor{black}{\sqrt{\EE\big[\|M_F-\Sigma\|_{\rm H.S.}^2\big]}},$$
with $M_F$ defined in {\rm (\ref{mf})}.
\end{theorem}  

As anticipated, to prove Theorem \ref{t}, we shall combine the somewhat classical smoothing estimate (\ref{smoothing}) with a remarkable bound by Schulte and Yukich \cite{SY19}.

\subsection{Applications}\label{ss:introapp}

We illustrate the use of Theorem \ref{t} by developing two examples in full detail.  

\medskip

\noindent\underline{\it Quantitative fourth moment theorems}. A {\bf fourth moment theorem} (FMT) is a mathematical statement implying that a given sequence of centered and normalized random variables converges in distribution to a Gaussian limit, as soon as the corresponding sequence of fourth moments converges to 3 (that is, to the fourth moment of the standard Gaussian distribution). Distinguished examples of FMTs are e.g. de Jong's theorem for degenerate $U$-statistics (see \cite{dJ, dp}) as well as the CLTs for multiple Wiener-It\^o integrals proved in \cite{nupe, pt}; the reader is referred to the webpage \cite{W} for a list (composed of several hundreds of papers) of applications and extensions of such results, as well as to the lecture notes \cite{Z} for a modern discussion of their relevance in mathematical physics.  Our first application of Theorem \ref{t} is a quantitative multivariate fourth moment theorem for a vector of multiple Wiener-It\^o integrals, considerably extending the qualitative multivariate results proved in \cite{pt}. Note that such a result was already obtained by Nourdin and Rosi\'nski \cite[Theorem 4.3]{NR14} for the 1-Wasserstein distance $d_W$. Thanks to Theorem \ref{t}, it is not difficult to generalize their result to the $\dc$ metric. 

\begin{corollary}\label{c:1}
Fix $m\geq 1$ as well as $q_1,\ldots,q_m\geq 1$.
Let $\mathbf F = (F_1,...,F_m)$ where $F_i = I_{q_{i}}(f_{q_i})$ with $f_{q_i}\in {\frak H}^{\odot q_i}$.  Let $N_\Sigma$ be a centered Gaussian vector with covariance matrix $\Sigma = (\EE F_iF_j)_{i,j\in [m]}$ \textcolor{black}{supposed to be invertible}. Then
\begin{align*}
\dc(\mathbf F, N_\Sigma) \le 402 \Big(\op{\Sigma^{-1}}^{3/2}+1\Big) m^{41/24} \sqrt{\EE \norm{\mathbf F}^4 - \EE \norm{N_\Sigma}^4}.
\end{align*}
\end{corollary}
In particular, for a vector ${\bf F}$ of multiple Wiener-It\^o integrals to be close in the convex distance to a centered Gaussian vector $N_\Sigma$ with matching covariance matrix, it is enough that $\EE\|{\bf F}\|^4\approx \EE \norm{N_\Sigma}^4$.

\medskip

\noindent\underline{\it The multivariate Breuer-Major theorem}. The second example concerns the convergence towards a Brownian motion occurring in the {\bf Breuer-Major theorem} proved in \cite{BM83}. Let us briefly recall this fundamental result (see \cite[Chapter 7]{NP} for an introduction to the subject, as well as \cite{nn, cnn} for recent advances in a functional setting).  Let $\{G_k : k\in\ZZ\}$ be a centered Gaussian stationary sequence with $\rho(j-k)=\EE[G_jG_k]$ and $\rho(0)=1$; in particular, $G_k\sim N(0,1)$ for all $k$. Let $\ph\in L^2(\RR,\ga)$ where $\ga(dx)=(2\pi)^{-1/2}e^{-x^2/2}dx$ denotes the standard Gaussian measure on $\R$.  Since the Hermite polynomials $\{H_k : k\ge 0\}$ form an orthonormal basis of  $L^2(\RR,\ga)$, one has 
\begin{align*}
\ph = \sum_{k\ge d} a_k H_k,
\end{align*}
with $d\in\NN$ and $a_d\neq 0$. The index $d$ is known as the Hermite rank of $\ph\in L^2(\RR,\ga)$.  Suppose in addition that $\int_\R \varphi d\gamma=\EE[\varphi(G_0)]=0$, that is, suppose $d\ge 1$. 
The Breuer-Major theorem \cite{BM83} states the following: if $\sum_{k\in\ZZ}|\rho(k)|^d<\infty$, then
\begin{align}\label{e:fdd}
\bra{\frac{1}{\sqrt{n}}\sum_{k=1}^{\lfloor nt\rfloor} \ph(G_k) : t\ge 0} \overset{f.d.d.}\longrightarrow  \bra{ \sigma W(t) : t\ge 0}
\end{align}
where $W$ is a standard Brownian motion, $\overset{f.d.d.}\longrightarrow$ indicates convergence in the sense of finite-dimensional distributions, and
\begin{align*}
\sigma^2 := \sum_{k\ge d}a_k^2 k! \sum_{j\in\ZZ} \rho(j)^k\in[0,\infty),
\end{align*}
(That $\sigma^2$ is a well-defined positive real number is part of the conclusion.)
We refer to our note \cite{NPY19} and references therein for results on the rate of convergence in the total variation distance for one-dimensional marginal distributions (that is, in dimension 1).  We intend to apply Theorem \ref{t} to address the rate of convergence for the following multivariate CLT implied by \eqref{e:fdd}: for every $0=t_0<t_1<...<t_m=T<\infty$, 
\begin{align*}
\left(\frac{1}{\sqrt{n}}\sum_{k=1}^{\lfloor nt_1\rfloor} \ph(G_k),..., \frac{1}{\sqrt{n}}\sum_{k=1}^{\lfloor nt_m\rfloor} \ph(G_k)\right) \overset{d}\longrightarrow N(0,\Sigma(t_1, ..., t_m))
\end{align*}
where $\overset{d}\longrightarrow$ indicates converges in distribution, and $N(0,\Sigma(t_1, ..., t_m))$ is a $m$-dimensional centred Gaussian vector with covariance $\Sigma(t_1, ..., t_m)$ having entries $\sigma^2 t_i\wedge t_j$, $i,j=1,...,m$. Notice that for any $m\times m$ invertible matrix $A$, 
\begin{align*}
\dc(\mathbf F, \mathbf G) = \dc(A\mathbf F, A\mathbf G).
\end{align*}
Therefore, choosing $A$ appropriately, it suffices to consider the vector $\mathbf F_n =(F_{n,1},...,F_{n,m})$ with
\begin{align*}
F_{n,i} = \frac{1}{\sqrt{n}}\sum_{k=\lfloor nt_{i-1}\rfloor + 1}^{\lfloor nt_i\rfloor} \ph(G_k), \quad \quad i\in[m]
\end{align*}
and obtain the rate of convergence for 
\begin{align}\label{e:glimit}
\mathbf F_n \overset{d}\longrightarrow  N(0, \sigma^2 \mathrm{Diag}(t_{1}-t_0,...,t_m-t_{m-1}))=:N_\Sigma.
\end{align}
The following result provides a quantitative version of this CLT with respect to the distance $\dc$.  Recall from \cite{NPY19} that the minimal regularity assumption over $\ph$ for obtaining rates of convergence via the Malliavin-Stein method is that $\ph\in\mathbb D^{1,4}(\R,\gamma)$, meaning that  $\ph$ is absolutely continuous and both $\ph$ and its derivative $\ph'$ belong to $L^4(\RR,\ga)$. We say that $\ph$ is 2-sparse if its expansion in Hermite polynomials does not have consecutive non-zero coefficients. In particular, even functions are 2-sparse.

\begin{corollary}\label{c:2} Let $\mathbf F_n$ and $N_\Sigma$ be given in  \eqref{e:glimit}. Suppose that $\ph\in\mathbb D^{1,4}(\R,\gamma)$ with Hermite rank $d\ge 1$. Then,
\begin{itemize}
\item[i)] There exists a constant $C$ depending only on $\ph, m, \Sigma$ such that for each $n\in\NN$, 
\begin{align*}
\dc(\mathbf F_n, N_\Sigma) \le C\sum_{i,j=1}^m |\Sigma(i,j) - \EE[F_iF_j]|+ Cn^{-\frac{1}{2}}  \left(\sum_{|k|< n} |\rho(k)|\right)^\frac{3}{2}. 
\end{align*}
\item[ii)] If $d=2$, $\ph$ is 2-sparse and $b\in[1,2]$, then there exists a constant $C$ depending only on $\ph, m, \Sigma$ such that for each $n\in\NN$, 
\begin{align*}
\dc(\mathbf F_n, N_\Sigma)\le C\sum_{i,j=1}^m |\Sigma(i,j) - \EE[F_iF_j]|+ C n^{-(\frac{1}{b}-\frac{1}{2})} \left( \sum_{|k|<n}|\rho(k)|^2\right)^{\frac{1}{2}} \left( \sum_{|k|< n}|\rho(k)|^b\right)^{\frac{1}{b}} .
\end{align*}
\item[iii)] If $d=2$, $\ph$ is 2-sparse, and $\sum_{k\in\ZZ}|\rho(k)|^2<\infty$, then as $n\to\infty$,
\begin{align*}
\dc(\mathbf F_n, N_\Sigma)\to 0. 
\end{align*}
\end{itemize}
\end{corollary}

The rest of the note is organized as follows. The proof of Theorem \ref{t} is given in Section \ref{pf:t}, Corollary \ref{c:1} in Section \ref{pf:c1}, Corollary \ref{c:2} in Section \ref{pf:c2}. We use $C$ to denote a generic constant whose value may change from line to line. 

\subsection{Acknowledgments}

\textcolor{black}{We thank Simon Campese and Nicola Turchi for pointing out an error in an earlier version}. 
I. Nourdin is supported by the FNR grant APOGee (R-AGR-3585-10) at Luxembourg University; G. Peccati is supported by the FNR grant FoRGES (R-AGR-3376-10) at Luxembourg University; X. Yang was supported by the FNR Grant MISSILe (R-AGR-3410-12-Z) at Luxembourg and Singapore Universities.

\section{Proofs} 
 
 \subsection{Proof of Theorem \ref{t}}\label{pf:t}
 We divide the proof into several steps.  
 
 \medskip
\noindent\underline{\textit{Step 1 (smoothing)}}. For any bounded and measurable $h$ and $t\in(0,1)$, recall its mollification at level $\sqrt{t}$ from (\ref{e:h_t}).
Then it is plain that $h_t$ is $C^\infty$ with bounded derivatives of all orders and  the solution to \eqref{e:SE} with $h=h_t$ is given by
\begin{align*}
f_t(\mathbf x) := -\frac{1}{2}\int_t^1 \frac{1}{1-s} (\EE[h(\sqrt{s}N_\Sigma+\sqrt{1-s}\mathbf x)]-\EE[h(N_\Sigma)]) ds,
\end{align*}
see \cite[p.12]{SY19}.
Finally, recall from e.g. \cite[Lemma 2.2]{SY19} that, for any $t\in(0,1)$,
\begin{align*}
\dc(\mathbf F,N_\Sigma) \le \frac{4}{3} \sup_{h\in\mathcal{I}_m} |\EE h_t(\mathbf F) -\EE h_t(N_\Sigma) | + \frac{20m}{\sqrt{2}}\frac{\sqrt{t}}{1-t}.
\end{align*}

 \noindent{\underline{\textit{Step 2 (integration by parts)}}.
 An integration by parts by \eqref{duality} and \eqref{chainrule} 
 (see \cite[Chapter 4]{NP} for more details), together with Cauchy-Schwarz's inequality, implies,  
 \begin{align}\label{e:line1}
|\EE h_{t}(\mathbf F) - \EE h_{t}(N_\Sigma)| &=  \left|\EE \sum_{i,j=1}^d \Sigma(i,j) \partial^2_{ij} f_t({\bf F})  - \EE \sum_{k=1}^d F_k \partial_k f_t(\mathbf F)\right | \\
&= \left|\EE \sum_{i,j=1}^d (\Sigma(i,j) - \langle DF_i, u_j\rangle) \partial^2_{ij} f({\bf F})\right| \nonumber\\
&\le \sqrt{\sum_{i,j=1}^d \EE[(\Sigma(i,j)-\langle DF_i, u_j\rangle)^2]} \sqrt{\sum_{i,j=1}^d \EE[\partial_{ij}^2f(\mathbf F)^2]}.\nonumber
\end{align}
The following remarkable estimate is due to M. Schulte and J. Yukich.
\begin{lemma}[Proposition 2.3 in \cite{SY19}]\label{l:SY2.4}Let $\mathbf{Y}$ be an $\RR^m$-valued random vector and $\Sigma$ be an invertible $m\times m$ covariance matrix. Then,
\begin{align*}
\sup_{h\in\mathcal I_m} \EE \sum_{i,j=1}^m |\partial^2_{ij} f_t(\mathbf Y)|^2 \le \op{\Sigma^{-1}}^2 \Big(m^2 (\log t)^2  \dc(\mathbf Y,N_\Sigma)+ 530 m^{17/6}\Big).
\end{align*}
where the left-hand side depends on $h$ through the function $f_t$ solving Stein's equation with test function $h_t$ given by \eqref{e:h_t}. 
\end{lemma}   
\begin{remark}
 Lemma \ref{l:SY2.4} improves upon the uniform bound (see \cite{CGS} or \cite{SY19})
\begin{align*}
 |\partial^2_{ij} f_t(\mathbf x) | &\le C(m,\Sigma) \norm{h}_\infty |\log t|,
 \end{align*} 
  when some a priori estimate on $\dc(\mathbf Y,N_\Sigma)$ is available.
\end{remark}
\noindent As consequence,
 \begin{align*}
 &|\EE h_{t}(\mathbf F) - \EE h_{t}(N_\Sigma)|   \\
 \leq & \op{\Sigma^{-1}} \Big(m |\log t| \dc(\mathbf F,N_\Sigma)^{1/2}+ 24 m^{17/12}\Big) 
  \sqrt{\sum_{i,j=1}^d \EE[(\Sigma(i,j)-\langle DF_i, u_j\rangle)^2]}.
 \end{align*}
 Letting 
 \begin{align*}
 \kappa&= \dc(\mathbf F,N_\Sigma),\\
 \ga &= \sqrt{\sum_{i,j=1}^d \EE[(\Sigma(i,j)-\langle DF_i, u_j\rangle)^2]},
 \end{align*}
 we have thus established 
 \begin{align}\label{e:kappa}
 \kappa \le  \frac{4}{3} \op{\Sigma^{-1}} (m|\log t| \sqrt{\kappa} + 24 m^{17/12}) \ga + \frac{20m}{\sqrt{2}}\frac{\sqrt{t}}{1-t}.
 \end{align}
 
 \noindent{\underline{\textit{Step 3 (exploiting the recursive inequality)}}. 
\textcolor{black}{If $\ga\ge 1/e$, then the bound we intend to prove holds trivially (observe that $\dc(\mathbf F,N_\Sigma)\le 1$ by definition). 
Without loss of generality we can and will therefore assume that $\ga \le 1/e$.}
Let $t=\ga^2$. Using the fact that $\kappa\le 1$ for the $\kappa$ on the right-hand side of the \eqref{e:kappa}, one has
 \begin{align*}
 \kappa \le \frac{4}{3}\op{\Sigma^{-1}}(2m|\log \ga| + 24m^{17/12})\ga + 20\sqrt{2}m \ga. 
\end{align*}  
Therefore,
\begin{align*}
|\log \ga| \sqrt{\kappa} & \le \op{\Sigma^{-1}}^{1/2} \sqrt{\frac{8m}{3}} \ga^{1/2} | \log \ga|^{3/2} + \left( \op{\Sigma^{-1}} 32m^{17/12} + 20\sqrt{2} m\right)^{1/2} \ga^{1/2} | \log\ga|.
\end{align*}
Since $\sup_{x\in(0, 1/e]} x^{1/2}| \log x |^{3/2}\le 4$, one has
\begin{align*}
|\log \ga| \sqrt{\kappa} &\le \frac{8\sqrt{6}}{3} \op{\Sigma^{-1}}^{1/2} m^{1/2} + 16\sqrt{2} \op{\Sigma^{-1}}^{1/2} m^{17/24} + 8\sqrt{10} m^{1/2} \\
&\le 58\left(\op{\Sigma^{-1}}^{1/2}+1\right) m^{17/24}.
\end{align*}
Hence, putting the estimate back into \eqref{e:kappa} with $t=\ga^2$ gives
\begin{align*}
\kappa&\le \frac{4}{3}\op{\Sigma^{-1}}\left( 2m\Big( 58(\op{\Sigma^{-1}}^{1/2}+1) m^{17/24}  \Big) + 24 m^{17/12} \right) \ga + 20\sqrt{2}m\ga \\
&\le \Bigg( \frac{4}{3}\times 140\times 2\Big(\op{\Sigma^{-1}}^{3/2}+1\Big) m^{41/24} + 20\sqrt{2}m\Bigg) \ga \\
&\le 402 \Big(\op{\Sigma^{-1}}^{3/2}+1\Big) m^{41/24} \ga. 
\end{align*}
The proof is complete. 

\subsection{Proof of Corollary \ref{c:1}}\label{pf:c1}

We will obtain the desired conclusion as a direct application of Theorem \ref{t} with $u_i=-DL^{-1}F_i$, see (\ref{e:relation}). 
Indeed, recall that by Step 2 of \cite[Proof of Theorem 4.3]{NR14}, for any $i,j\in[m]$, 
\begin{align*}
\EE[(\EE[F_iF_j] - \langle DF_i, -DL^{-1}F_j\rangle )^2]\le \Cov(F_i^2,F_j^2) - 2\EE [F_i F_j]^2.
\end{align*}
On the other hand, Step 3 of \cite[Proof of Theorem 4.3]{NR14} shows that
\begin{align*}
\sum_{i,j=1}^m \Cov(F_i^2,F_j^2) - 2\EE [F_i F_j]^2 = \EE \norm{\mathbf F}^4  -\EE \norm{N_\Sigma}^4. 
\end{align*}
Plugging these estimates  into Theorem \ref{t} gives the result. 

\subsection{Proof of Corollary \ref{c:2}}\label{pf:c2}
We follow closely the arguments of \cite{NPY19} and assume without loss of generality that $T=1$. First, one can embed the Gaussian sequence in the statement in an isonormal Gaussian process $\{X(h) : h\in\frak H \}$, in such a way that 
\begin{align*}
\bra{G_k : k\in\ZZ} \overset{d}= \bra{ X(e_k) : k\in\ZZ  },
\end{align*}
for some appropriate family $\{e_k\}\subset \frak H$ verifying $\langle e_j , e_k \rangle_{\frak H} = \rho(k-j)$ for all $j,k$. For $\ph=\sum_{\ell\ge d}a_\ell H_\ell\in L^2(\R,\gamma)$, we define the shift mapping $\ph_1 := \sum_{\ell\ge 1}a_\ell H_{\ell-1}$ and set
\begin{align*}
u_{n,i} := \frac{1}{\sqrt{n}} \sum_{m=\lfloor nt_{i-1} \rfloor+1}^{\lfloor nt_i \rfloor} \ph_1(G_m) e_m, \quad i\in[m].
\end{align*}
Then, standard computations using \eqref{e:relation}  \textcolor{black}{lead to}
\begin{align}\label{e:key}
\de(u_{n,i}) =F_{n,i}.
\end{align}
Applying Theorem \ref{t} and the triangle inequality implies that
\begin{align*}
\dc(\mathbf F, N_\Sigma) \le C \sum_{i,j=1}^m |\Sigma(i,j) - \EE[F_iF_j]| + C \sqrt{ \sum_{i,j=1}^m \Var(\langle DF_{n,i}, u_{n,j}\rangle )}=: I_1+I_2.
\end{align*}
Note that, by the chain rule and the relation $D(G_k) =e_k$, 
\begin{align*}
\langle DF_{n,i}, u_{n,j}\rangle_\frak{H} = \frac{1}{ n} \sum_{k\sim t_i}\sum_{\ell \sim t_j } \ph'(G_k)\ph_1(G_\ell) \rho(k-\ell),
\end{align*}
where $k\sim t_i$ means that the sum is taken over $k\in \{ \lfloor nt_{i-1}\rfloor +1,..., \lfloor nt_i\rfloor \}$, and similarly for the symbol $\ell\sim t_j$. Hence,
\begin{eqnarray}
\label{e:THE_variance}
&&\Var(\langle DF_{n,i}, u_{n,j}\rangle_\frak{H})\notag\\
& =&\!\!\! \frac{1}{n^2} \sum_{k\sim t_i}\sum_{\ell \sim t_j }\sum_{k'\sim t_i}\sum_{\ell' \sim t_j }  \Cov(\ph'(X_k)\ph_1(X_\ell), \ph'(X_{k'})\ph_1(X_{\ell'}))\rho(k-\ell)\rho(k'-\ell')\notag\\
&\le & \frac{1}{n^2}  \sum_{k,k',\ell,\ell'=1}^{n} \textcolor{black}{\Big|} \Cov(\ph'(X_k)\ph_1(X_\ell), \ph'(X_{k'})\ph_1(X_{\ell'}))\rho(k-\ell)\rho(k'-\ell')\textcolor{black}{\Big|} .\\
\notag
\end{eqnarray}
The variance is bounded because of the assumption that $\ph\in\mathbb D^{1,4}$. Once \eqref{e:THE_variance} is in place, one can apply Gebelein's inequality as in \cite{NPY19}. In particular, one infers that (see \cite[Proposition 3.4]{NPY19})
\begin{align*}
\dc(\mathbf F_n, N_\Sigma) \le C\sqrt{\frac{1}{n^2} \sum_{i,j,k,\ell=0}^{n-1} 
\bigg| \rho(j-k) \rho(i-j)\rho(k-\ell)\bigg|}. 
\end{align*}
If, in addition, $\ph$ is $2$-sparse, then
\begin{align*}
\dc(\mathbf F_n, N_\Sigma) \le C\sqrt{\frac{1}{n^2} \sum_{i,j,k,\ell=0}^{n-1}  \bigg|\rho(j-k)^2 \rho(i-j)\rho(k-\ell)\bigg| }.
\end{align*}
Items i)-ii) now follow from these inequalities, as shown in \cite{NPY19}; we include a  proof for completeness.  
Applying twice Young's inequality for convolutions, one has, \textcolor{black}{with $\rho_n(k)=|\rho(k)|{1}_{|k|<n}$},
\begin{eqnarray*}
\sum_{i,j,k,\ell=0}^{n-1}  \bigg|\rho(j-k) \rho(i-j)\rho(k-\ell)\bigg|  &\le& \sum_{i,\ell=0}^{n-1} \big(\rho_n*\rho_n*\rho_n\big)(i-\ell)\\ 
&\le& n \norm{\rho_n*\rho_n*\rho_n}_{\ell^1(\ZZ)} \le n \norm{\rho_n}^3_{\ell^1(\ZZ)}, 
\end{eqnarray*}
yielding  Item i).  Rewrite the sum of products as a sum of the product of convolutions by introducing the function 
$1_n(k):= 1_{|k|<n}$. We have
\begin{align*}
&\sum_{i,j,k,\ell=0}^{n-1}  |\rho(j-k)^2 \rho(i-j)\rho(k-\ell)| \\
 &= \sum_{i,j,k,\ell=0}^{n-1}  |\rho(j-k)^2 \rho(i-j)\rho(k-\ell) 1_n(\ell-i)| \\ 
&\textcolor{black}{\le} \sum_{j,\ell=0}^{n-1}  (\rho_n* 1_n)(\ell-j) (\rho_n * \rho_n^2)(\ell-j) \le n \langle \rho_n*1_n, \rho_n*\rho^2_n\rangle_{\ell^2(\ZZ)}.
\end{align*}
For $b\in[1,2]$, we have
\begin{align}\label{e:11}
\langle \rho_n*1_n, \rho_n*\rho^2_n\rangle_{\ell^2(\ZZ)} &\le \norm{\rho_n*1_n}_{\ell^{b\over {b-1}}(\ZZ)} \norm{\rho_n*\rho^2_n}_{\ell^b(\ZZ)} \\
&\le \norm{\rho_n}_{\textcolor{black}{\ell^{b}(\ZZ)}} \norm{1_n}_{\ell^{\frac{b}{2b-2}}(\ZZ)} \norm{\rho_n}_{\ell^{b}(\ZZ)} \norm{\rho_n^2}_{\ell^{1}(\ZZ)} \notag\\
&\textcolor{black}{\le (2n)^{\frac{2b-2}{b}}} \norm{\rho_n^2}_{\ell^{1}(\ZZ)} \norm{\rho_n}^2_{\ell^b(\ZZ)},\notag
\end{align}
yielding Item ii). Now we move to the proof of Item iii). Notice that taking $b=2$ for the right-hand side of \eqref{e:11}, together with an application of Young's inequality,  yields that
\begin{align*}
\langle \rho_n*1_n, \rho_n*\rho^2_n\rangle_{\ell^2(\ZZ)} \le \norm{\rho_n*1_n}_{\ell^{2}(\ZZ)} \norm{\rho_n}^3_{\ell^2(\ZZ)}.
\end{align*}
Thus, 
\begin{align*}
\frac{1}{n^2} \sum_{i,j,k,\ell=0}^{n-1}  \bigg|\rho(j-k)^2 \rho(i-j)\rho(k-\ell)\bigg|  \le \frac{1}{n} \norm{\rho_n*1_n}_{\ell^{2}(\ZZ)}\norm{\rho_n}^3_{\ell^2(\ZZ)}. 
\end{align*}
To proceed, we handle the convolution involving $1_n$ a bit differently. Set
\begin{align*}
\wt \rho_n(k) &= \rho(k) 1_{N\le |k|<n},\\
\wh\rho_n(k) &=  \rho(k) 1_{|k|\le N}
\end{align*}
so that $\rho_n=\wt \rho_n + \wh\rho_n$. One has
\begin{align*}
\frac{1}{n} \norm{\rho_n*1_n}_{\ell^{2}(\ZZ)} &\le \frac{1}{n}\norm{\wt \rho_n}_{\ell^2(\ZZ)} \norm{1_n}_{\ell^1(\ZZ)} +  \frac{1}{n} \norm{\wh\rho_n}_{\ell^1(\ZZ)} \norm{1_n}_{\ell^2(\ZZ)} \\
&\le \Big(\sum_{N\le |k|< n} \rho(k)^2\Big)^{1/2} + (2N+1)n^{-1/2},
\end{align*}
from which Item iii) follows. The proof is complete.

\appendix
{\blue
\section{Proof and discussion of relation \eqref{e:conwass}}\label{s:conwass}

Inequality \eqref{e:conwass} is a direct consequence of the following statement, whose proof exploits a strategy already adopted in \cite[Proof of Theorem 3.1]{app}. 

\begin{p}\label{p:conwass} Fix $m\geq 1$, and let $N_\Sigma$ denote a $m$-dimensional centered Gaussian vector with invertible covariance matrix $\Sigma$. Then, for any $m$-dimensional random vector ${\bf F}$ one has that 
\begin{equation}\label{e:ineqq}
d_c({\bf F} , N_{\Sigma})\leq 2 \sqrt{2} \, \Gamma(\Sigma)^{1/2} \, d_W({\bf F} , N_{\Sigma})^{1/2},
\end{equation}
where $\Gamma(\Sigma)$ is the isoperimetric constant defined by
$$
\Gamma(\Sigma) := \sup_{Q, \epsilon>0} \frac{ \mathbb{P}(N_\Sigma \in Q^\epsilon) - \mathbb{P}(N_\Sigma \in Q)} {\epsilon} ,
$$
where $Q$ ranges over all Borel measurable convex subsets of $\R^m$, and $Q^{\epsilon}$ indicates the set of all elements of $\R^m$ whose Euclidean distance from $Q$ does not exceed $\epsilon$.
\end{p}
\begin{remark}{\rm In \cite{nazarov} it is proved that, for some absolute constants $0<c<C<\infty$, 
$$
c\sqrt{\|\Sigma\|_{\rm H.S.} } \leq \Gamma(\Sigma) \leq C\sqrt{\|\Sigma\|_{\rm H.S.} },
$$
where $\|\cdot \|_{\rm H.S.}$ stands as above for the Hilbert-Schmidt norm. When $\Sigma= I_m$ (identity matrix), one has also the well-known estimate $\Gamma(I_m) \leq 4m^{1/4}$ (see \cite{ball}), as well as Nazarov's upper and lower bounds
$$
 e^{-5/4} \leq \liminf_m \frac{\Gamma(I_m)}{m ^{1/4}}\leq \limsup_m \frac{\Gamma(I_m)}{m^{1/4}}\leq (2\pi)^{-1/4} < 0.64,
$$ 
see \cite[p. 170]{nazarov}. In \cite[Theorem 1.2]{R19}, it is proved that Nazarov's upper bound can be reduced from 0.64 to 0.59; see also \cite{bentkus} for related computations in the framework of the multivariate CLT.
}\end{remark}

\begin{proof}[Proof of Proposition \ref{p:conwass}] We can assume that ${\bf F}$ and $N_{\Sigma}$ are defined on a common probability space, and that $\mathbb{E} \| {\bf F} - N_{\Sigma}\|_{\R^m} = d_W({\bf F}, N_{\Sigma})$. Fix a convex set $Q$, as well as $\epsilon >0$. We have that 
\begin{eqnarray*}
\mathbb{P}[ {\bf F} \in Q ] - \mathbb{P}[ {N_\Sigma}   \in Q]   &\leq&  \mathbb{P}[ {\bf F} \in Q,  \| {\bf F} - N_{\Sigma}\|_{\R^m} \leq \epsilon ] - \mathbb{P}[N_\Sigma \in Q]  + \epsilon^{-1} \mathbb{E}[ \| {\bf F} - N_{\Sigma}\|_{\R^m}]  \\
&\leq &  \mathbb{P}[ N_{\Sigma} \in Q^{\epsilon} ] - \mathbb{P}[N_\Sigma \in Q] + \epsilon^{-1} d_W ({\bf F}, N_{\Sigma})\\
&\leq & \Gamma(\Sigma)\epsilon + \epsilon^{-1} d_W ({\bf F}, N_{\Sigma}).
\end{eqnarray*}
On the other hand, defining $Q^{-\epsilon}$ as the set of those $y\in Q$ such that the closed ball with radius $\epsilon$ centered at $y$  is contained in $Q$,
\begin{eqnarray*}
\mathbb{P}[ N_\Sigma \in Q ] - \mathbb{P} [{\bf F}\in Q]   &\leq&  \mathbb{P}[ N_\Sigma \in Q,  \| {\bf F} - N_{\Sigma}\|_{\R^m} \leq \epsilon ] - \mathbb{P}[{\bf F} \in Q^{-\epsilon} ]  + \epsilon^{-1} \mathbb{E}[ \| {\bf F} - N_{\Sigma}\|] 
\\
&\leq & \Gamma(\Sigma)2\epsilon + \epsilon^{-1} d_W ({\bf F}, N_{\Sigma}),
\end{eqnarray*}
where we have used the inequality
$$
\mathbb{P}[ N_\Sigma \in Q,  \| {\bf F} - N_{\Sigma}\|_{\R^m} \leq \epsilon ] - \mathbb{P}[{\bf F} \in Q^{-\epsilon} ] \leq \mathbb{P} [N_\Sigma \in Q] -\mathbb{P} [N_\Sigma \in Q^{-2\epsilon} ].   
$$
The conclusion follows from a standard optimisation in $\epsilon$.

\end{proof}

\begin{remark}{\rm Fix $m\geq 1$, and let $\mathscr{R}_m$ be the collection of all hyper-rectangles of the type $R = (-\infty, t_1]\times \cdots \times (-\infty, t_m]$. In \cite[Theorem 3.1]{app} it is proved that, if $N$ is a $m$-dimensional centered Gaussian vector with identity covariance matrix and ${\bf F}$ is any $m$-dimensional random vector, then
\begin{equation}\label{e:kd}
\sup_{R\in \mathscr{R}_m}  \big| \mathbb{P}[ {\bf F} \in R ] - \mathbb{P}[ {N}   \in R] \big| \leq 3  ( \log m ) ^{1/4}  d_W({\bf F} , N)^{1/2}.
\end{equation}
The left-hand side of the previous inequality is usually referred to as the {\bf Kolmogorov distance} between the distributions of ${\bf F}$ and $N$. The presence of the factor $(\log m)^{1/4}$ is consistent with the fact that, for the standard Gaussian measure on $\RR^m$, the isoperimetric constant associated with all hyper-rectangles of $\RR^m$ is bounded from above by $\sqrt{\log m}$, see \cite{ball, nazarov}. An estimate analogous to \eqref{e:kd} is established by different methods in \cite[Corollary 3.1]{koike}.

}
\end{remark}

}

\bibliographystyle{plain}

\begin{thebibliography}{123}

\bibitem{W} Malliavin-Stein approach: a webpage maintained by Ivan Nourdin. Address: \\ {\tt https://sites.google.com/site/malliavinstein/home}

\bibitem{app} E. Azmoodeh, G. Peccati and G. Poly (2016): The law of iterated logarithm for subordinated Gaussian sequences: uniform Wasserstein bounds. {\it ALEA} {\bf 13}, 659-686.

\bibitem{ball} K. Ball (1993): The reverse isoperimetric problem for the Gaussian measure. {\it Discrete Comput. Geom.}, {\bf 10} , no. 4, 4111-420.


\bibitem{bentkus} V. Bentkus (2003): On the dependence of the Berry–Esseen bound on dimension. {\it Journal of Statistical Planning and
Inference} {\bf 113}, 385--402. 

\bibitem{BM83} P. Breuer and P. Major (1983): Central limit theorems for nonlinear functionals of Gaussian fields. {\it J. Multivariate Anal.} {\bf 13} , no. 3, pp. 425--441. 


\bibitem{CGS} L.H.Y. Chen, L. Goldstein and Q.-M. Shao (2011): {\it Normal approximation by Stein's method}. Probability and its Applications (New York). Springer, Heidelberg. 


\bibitem{cnn} S. Campese, I. Nourdin and D. Nualart (2020): Continuous Breuer-Major theorem: tightness and non-stationarity. \textcolor{black}{{\it Ann. Probab.} {\bf 48}, no. 1, pp. 147-177}. 

\bibitem{dJ} P. de Jong (1990): A central limit theorem for generalized multilinear forms. {\it J. Multivariate Anal.}
{\bf 34}(2), 275--289.

\bibitem{dp} Ch. D\"obler and G. Peccati (2017): Quantitative de Jong theorems in any dimension. {\it Electron. J.
Probab.} {\bf 22}(2), 1--35.

\bibitem{G} F. G\"otze (1991): {On the rate of convergence in the multivariate CLT}. {\it Ann. Probab.}, {\bf 19}(2), 724--739.

\bibitem{KP15} Y.T. Kim and H.S. Park (2015): Kolmogorov distance for multivariate normal approximation. {\it Korean J. Math.} {\bf 23} , no. 1, pp. 1--10.

\bibitem{koike} Y. Koike (2019): High-dimensional central limit theorems for homogeneous sums. Preprint.

\bibitem{nagaev} S.V. Nagaev (1976). An estimate of the remainder term in the multidimensional
central limit theorem. In: {\it Proceedings of the Third Japan–USSR Symposium on
Probability Theory (Tashkent, 1975)}. Lecture Notes in Math. {\bf 550}, 419--438. Berlin:
Springer.

\bibitem{nazarov} F. Nazarov (2003): On the Maximal Perimeter of a Convex Set in $\mathbb{R}^n$ with Respect to a Gaussian Measure. In: Milman V.D., Schechtman G. (eds) {\it Geometric Aspects of Functional Analysis}. Lecture Notes in Mathematics, vol. {\bf 1807}, Springer, Berlin, Heidelberg, 169–187.

\bibitem{nn} I. Nourdin and D. Nualart (2020): The functional Breuer-Major theorem. \textcolor{black}{{\it Probab. Th. Related Fields} {\bf 176}, pp. 203-218.}

\bibitem{NP09} I. Nourdin and G. Peccati (2009):  Stein's method on Wiener chaos. {\it Probab. Th. Related Fields} {\bf 145}, pp. 75--118.

\bibitem{NP} I. Nourdin and G. Peccati (2012):  {\it Normal Approximations with Malliavin calculus. From Stein's Method to Universality}. Cambridge Tracts in Mathematics, 192. Cambridge University Press, Cambridge. xiv+239 pp.

\bibitem{NPRei10} I. Nourdin, G. Peccati and G. Reinert (2010): Invariance principles for homogeneous sums: universality of Gaussian Wiener chaos. {\it Ann. Probab.}{\bf  38}(5), 1947--1985.

\enlargethispage{0.8cm}

\bibitem{NPR10} I. Nourdin, G. Peccati and A. R\'eveillac (2010): Multivariate normal approximation using Stein's method and Malliavin calculus. {\it Ann. Inst. Henri Poincar\'e Probab. Stat.} {\bf 46}(1), pp. 45--58.

\bibitem{NPS14} I. Nourdin, G. Peccati and Y. Swan (2014): Entropy and the fourth moment phenomenon. {\it J. Funct. Anal.} {\bf 266} , no. 5, pp. 3170--3207.

\bibitem{NPY19} I. Nourdin, G. Peccati and X. Yang (2019): Berry-Esseen bounds in the Breuer-Major CLT and Gebelein's inequality. {\it Electron. Commun. Probab.} {\bf 24}, Paper No. 34, 12 pp. 

\bibitem{NR14} I. Nourdin and J. Rosi\'nski (2014): Asymptotic independence of multiple Wiener-It\^o integrals and the resulting limit laws. {\it Ann. Probab.} {\bf 42}, no. 2, pp. 497--526.


\bibitem{nualartbook}
\rm D. Nualart (2006).
\it The Malliavin calculus and related topics.
\rm Springer-Verlag, Berlin, second edition.


\bibitem{eulalia}
\rm D. Nualart and E. Nualart (2018).
\em  Introduction fo Malliavin calculus.
\em  Cambridge University Press.

\bibitem{nupe} D. Nualart and G. Peccati (2005): Central limit theorems for sequences of multiple stochastic integrals, {\it Ann. Probab.} {\bf 33}(1), 177--193.



\bibitem{pt} G. Peccati and C.A. Tudor (2004): Gaussian limits for vector-valued multiple stochastic integrals. {\it S\'eminaire de Probabilit\'es XXXVIII}, 247--262.


\bibitem{PZ} G. Peccati and C. Zheng (2010): Multi-dimensional Gaussian fluctuations on the Poisson
space. {\it Electron. J. Probab.} {\bf 15}, 1487--1527.


\bibitem{R19} M. Raic (2019): A multivariate Berry–Esseen theorem with explicit constants. {\it Bernoulli}
{\bf 25}(4A), 2824--2853.


\bibitem{SY19} M. Schulte and J.E. Yukich (2019):  Multivariate second order Poincar\'e
inequalities for Poisson functionals. {\it Electron. J. Probab.} {\bf 24}, no. 130, pp. 1--42.

\bibitem{V} C. Villani (2009). {\it Optimal transport. Old and new.} Grundlehren der mathematischen Wissenschaften {\bf 338}. Springer: Berlin.

\bibitem{Z} N. Zygouras (2019). Discrete stochastic analysis. Lecture notes available on the webpage:\\ {\tt https://warwick.ac.uk/fac/sci/statistics/staff/academic-research/zygouras/}

\end{thebibliography}

\end{document}